\documentclass{article}

\usepackage[english]{babel}

\usepackage[letterpaper,top=2cm,bottom=2cm,left=3cm,right=3cm,marginparwidth=1.75cm]{geometry}

\usepackage{amsmath}
\usepackage{amssymb}
\usepackage{graphicx}
\usepackage[colorlinks=true, allcolors=blue]{hyperref}
\usepackage{amsthm}
\newtheorem{theorem}{Theorem}
\newtheorem{Proposition}[theorem]{Proposition}
\newtheorem{lemma}{Lemma}
\theoremstyle{remark}
\newtheorem{remark}{Remark}

\title{A logarithmic convexity approach to quantitative unique continuation for the complex Ginzburg-Landau operator}
\author{Yueliang Duan\thanks{Department of Mathematics, Shantou University, Shantou 515063, China;
e-mail: ylduan@stu.edu.cn. This work was supported by the National Natural Science Foundation of China under grant 12201379, the Guangdong Basic and Applied Basic Research Foundation (2026A1515012322) and the Guangdong Provincial Department of Education (2025KCXTD013).} ,
\quad Borui Liu
\thanks{School of Mathematics and Statistics,
Wuhan University, Wuhan 430072, China; e-mail: lbr\_1419123271@163.com} ,
\quad Can Zhang
\thanks{School of Mathematics and Statistics,
Wuhan University, Wuhan 430072, China;
e-mail: canzhang@whu.edu.cn. This work was supported by the
National Natural Science Foundation of China under grant 12422118.}}

\begin{document}
\maketitle

\begin{abstract}
We establish quantitative unique continuation estimates for solutions of the complex Ginzburg-Landau equation in the framework of two-sphere and one-cylinder inequalities. While prior work \cite{Dou-Fu-Liao-Zhu} relied on Carleman estimates, the present study develops a novel logarithmic convexity approach to prove the quantitative unique continuation property. We derive an explicit quantitative unique continuation constant and fully characterize its dependence on the parameters.
\end{abstract}

\medskip


\section{Introduction and the main result}

Unique continuation is one of the fundamental properties in partial differential equation theory. A large body of work has addressed quantitative unique continuation for parabolic equations, with key contributions from \cite{Escauriaza,P-W,POON,Zhu}. Such research aims to establish quantitative bounds that link the local profile of a solution to its global behavior.

The global Carleman estimate method and the spectral method are two important approaches for studying the quantitative unique continuation property of parabolic equations. The Carleman estimate method, first proposed by Fursikov and Imanuvilov \cite{ikov}, adopts exponentially decaying or growing weight functions to amplify the solution’s behavior in target domains. The spectral method, invented by Lebeau and Robbiano \cite{Lebeau-Robbiano}, utilizes the concept of infinite-step iteration together with the exponential decay of solutions to the equation. Additional relevant literature can be found in \cite{Escauriaza-Fernandez,ikov,Rousseau-Robbiano}.
Besides the Carleman estimate method and the spectral method, the frequency function method and the propagation of smallness method are also important approaches for unique continuation analysis. The frequency function method establishes quantitative unique continuation via the monotonicity property of the frequency function, which converts local smallness of solutions into global control. We refer to \cite{L-Y,Phung-Wang,P-W,POON} for its applications to parabolic equations. Over recent years, the propagation of smallness has grown into a major research branch in quantitative unique continuation for parabolic equations. Generalizing from H$\ddot{\mathrm{o}}$lder-type propagation results for elliptic equations, this approach circumvents the analyticity requirement. For further details, see \cite{B-M1,B-M2,Duan-Yu-Zhang,Logunov-Malinnikova}.

The logarithmic convexity method is a powerful tool for analyzing quantitative unique continuation for parabolic equations. It relies on constructing an energy-associated logarithmically convex function in spacetime. Using the inherent parabolic structure, one can verify the convexity of its logarithm with respect to a prescribed parameter. Such convexity ensures stable propagation of the solution’s local smallness across the global domain, thereby deriving the desired quantitative unique continuation estimates (see, e.g., \cite{Bardos-Phung,Buffe-Phung,K-D-Phung}).
In this work, we develop a refined logarithmic convexity method to establish quantitative unique continuation results for the Ginzburg-Landau operator.

The Ginzburg-Landau operator originated from a phenomenological model for superconductivity introduced in \cite{El}. At present, it serves to model a broad range of physical systems: superconductivity, superfluidity, Bose-Einstein condensation, liquid crystals, field-theoretic cosmic strings and instability waves (see, e.g., \cite{Aranson, Rosenstein}). It ranks among the most thoroughly investigated nonlinear operators in both physics and mathematics. We refer to \cite{Guo-Jiang-Li} for a comprehensive account of its well-posedness and other fundamental properties.
Controllability analysis of the Ginzburg-Landau equation has become an active research topic. Most existing studies, such as \cite{Dou-Fu-Liao-Zhu,fu3,rosier}, employ global Carleman estimates. To the best of our knowledge, the frequency function method has not yet been utilized to prove quantitative unique continuation for this equation. The primary challenge is the construction of proper weight functions.

Our main focus is on solutions to linear Ginzburg-Landau equations of the form
\begin{equation}\label{1.1}
\begin{array}{lll}
\partial_{t}\varphi-(a+ib)\Delta\varphi+V(x,t)\varphi=0 &\mathrm{in}\ Q_{4}=B_{4}(0)\times(0,16),
\end{array}\end{equation}
where $B_{4}(0)\subset \mathbb R^{N}$ (with$N\geq1$) is the open ball with center at the origin and radius, and $a>0$, $b\in\mathbb R$, $V(\cdot, \cdot)\in L^\infty(\mathbb{R}^{N}\times(0,T);\mathbb C)$ and $i:=\sqrt{-1}$ denotes the imaginary unit.

The main result of this paper concerning the quantitative estimate of unique continuation for solutions of \eqref{1.1} is stated as follows.
\begin{theorem}\label{Thm1}
Let $\varphi$ be a non-trivial solution to \eqref{1.1}.
Then there are two positive constants $C:= C(r)$ and $\gamma:=\gamma(r)\in (0,1)$ so that for any $T\in (0,16]$ and for any $r\in(0,1/2)$,
\begin{eqnarray*}
\int_{B_{2r}}|\varphi(x,T)|^{2}\mathrm{d}x\leq e^{C\left(K_1
+K_2T^{-1}+\|V\|^{2/3}_{{\infty}}+T\|V\|_{\infty}\right)}\left(\int_{B_r} |\varphi(x,T)|^{2}\mathrm{d}x\right)^{1-\gamma}\left({\int_{T/2}^{T}\int_{B_{4r}}}|\varphi(x,s)|^{2}\mathrm dx\mathrm ds\right)^{\gamma}.
\end{eqnarray*}
where $B_{2r}:=B_{2r}(0)$, $K_{1}:=1+a^3+b^2+b^3/a^2+b^{6}/a^{3}$ and $K_{2}:=
1+|b|/a+a/(a^2+b^2)$.
\end{theorem}
\begin{remark}
It is worth noting that by taking $a=1$ and $b=0$ in \eqref{1.1}, the equation \eqref{1.1} reduces to the heat equation with a bounded potential. In the study of the heat equation, Lemma 3.2 in \cite{Duan-Wang-Zhang} provides a result analogous to Theorem \ref{Thm1} of the present paper. This shows that our work actually extends the quantitative unique continuation property of the heat equation.
\end{remark}

\begin{remark}
When $a=0$ and $b=1$, the equation \eqref{1.1} reduces to the Schr$\ddot{\mathrm{o}}$dinger equation. A unique continuation inequality for the Schr$\ddot{\mathrm{o}}$dinger equation is provided in Theorem 1.3 of \cite{WGS-WM-ZYB}. However, this unique continuation inequality is not of the two-sphere and one-cylinder type: on one hand, the spatial domain is the whole space $\mathbb{R}^N$ instead of a cylinder; on the other hand, the $L^2$ norm of $u_0$ is replaced by the $L^2$ norm of $e^{a|x|}u_0$. Indeed, in the present paper, when we set $b=1$ and let $a \to 0$, the quantitative uniqueness constant tends to infinity. These two facts suggest that, for the Schr$\ddot{\mathrm{o}}$dinger equation, a two-sphere and one-cylinder inequality may not hold.
\end{remark}

\begin{remark}
Using a global Carleman estimate approach, the authors of \cite{Dou-Fu-Liao-Zhu} obtained results similar to those presented in this paper. For dimensions $N=2,3$ and under suitable conditions on parameters $a$ and $b$, the observability is established (see Theorem 4.1 therein).
In contrast, this paper employs a log-convexity method to derive an observability inequality for equation \eqref{1.1} at a specific time.
Our approach relaxes the dimension requirement to $N\geq1$ and imposes no restrictions on the parameters $a$ and $b$. It should be noted that the work \cite{Dou-Fu-Liao-Zhu} deals with a nonlinear Ginzburg-Landau equation, whereas the method presented here is currently only applicable to the linear Ginzburg-Landau equation. Unfortunately, for the nonlinear case, we have not yet found a suitable approach for investigation.
\end{remark}

The structure of the rest of the paper is as follows. Section \ref{Aux} is devoted to two local energy estimates and a related proposition that is essential for the proof of Theorem \ref{Thm1}. Section \ref{pre} shows some elementary lemmas related to the log-convexity method. Section \ref{pr3} gives the proof of Theorem \ref{Thm1}. In Section \ref{Conclusions}, we present a summary and discuss some applications of Theorem \ref{Thm1} in the control theory of the Ginzburg-Landau equation. In the appendix, we provide the detailed proofs of Lemmas \ref{lemma-1.1} and \ref{lemma-1.2}.

\section{Auxiliary lemmas}\label{Aux}

In this section, we first show two standard energy estimates (see Lemmas \ref{lemma-1.1} and \ref{lemma-1.2} below). For the sake of completeness, we provide their detailed proofs in the appendix.

\begin{lemma}\label{lemma-1.1}
There is a constant $C_{1}>1$ so that
the solution  $\varphi$ of (\ref{1.1}) satisfies
\begin{equation}\label{1.2}
\begin{array}{lll}
&&\displaystyle{\max_{t\in[T-\tau_{1},T]}}\int_{B_{r}(0)}|\varphi(x,t)|^{2}\mathrm dx
+a\int_{T-\tau_{1}}^{T}\int_{B_{r}(0)}| \nabla\varphi(x,s)|^{2}\mathrm dx\mathrm ds\\
\\
&\leq& C_1 \left[\left(2a+\frac{|b|}{a}\right)(R-r)^{-2}+(\tau_{2}-\tau_{1})^{-1}+\|V\|_{\infty}\right]
\displaystyle{\int_{T-\tau_{2}}^{T}\int_{B_{R}(0)}}|\varphi(x,s)|^{2}\mathrm dx\mathrm ds,
\end{array}
\end{equation}
with $0<r<R\leq4$, $0<\tau_{1}<\tau_{2}<T\leq16$.
\end{lemma}

\begin{lemma}\label{lemma-1.2}
There is a constant $C_{2}>0$ so that
the solution $\varphi$ of (\ref{1.1}) satisfies
\begin{equation}\label{1.3}
\begin{array}{lll}
&&\displaystyle{\max_{t\in[T-\tau,T]}}\int_{B_{R}(0)}| \nabla\varphi(x,t)|^{2}\mathrm{d}x\\
\\
&\leq & \displaystyle{}C_{2}\left[1+a^{-1}\tau^{-2}+(a^2+b^2)^{2}a^{-3}R^{-4}+a^{-1}\|V\|^{2}_{\infty}\right]
\int_{T-2\tau}^{T}\int_{B_{2R}(0)}|\varphi(x,s)|^{2}
\mathrm{d}x\mathrm{d}s,
\end{array}
\end{equation}
with $0<R\leq4$, $0<\tau<T/2\leq8$.
\end{lemma}

To establish our main result, we require the following auxiliary lemma, which is inspired by Lemma 3 in \cite{Phung-Wang-Zhang}.
\begin{Proposition}\label{lemma-1.3}
Let $0<2r\leq R<+\infty$ and $\delta\in (0,1]$.
Then there are two constants $C_{3}:= C_{3}(r,\delta)>0$ and
$C_{4}:=C_{4}(r,\delta)>0$ so that for any  $0<\tau_{1}<\tau_{2}<T$,
the quantity
\begin{equation}\label{1.8}
\begin{array}{lll}
h_{0}=\frac{\hat{C}_{3}}{\ln\left[(1+\hat{C}_{4})\left(e^{[1+2C_{1}(1+\frac{2a+|b|/a}{r^2})]
(1+\frac{1}{\tau_{2}-\tau_{1}}+\|V\|^{2/3}_{\infty})+\frac{4\hat{C}_{3}}{T}
+2T\|V\|_{\infty}}\right)\frac{\int_{T-\tau_{2}}^{T}\int_{B_{R}(0)}
|\varphi(x,t)|^{2}\mathrm{d}x\mathrm{d}t}{\int_{B_{r}(0)}|\varphi(x,T)|^{2}\mathrm{d}x}\right]}
\end{array}
\end{equation}
(where $\hat{C}_{3}:=(a+b^{2}/a)^{-1}C_{3}$, $\hat{C}_{4}:=(a+b^{2}/a)C_{4}$ and  $C_{1}>1$ is the constant given by Lemma~\ref{lemma-1.1}),
 has the following two properties:
  \begin{description}
\item[($i$)] \begin{equation}\label{1.9}
    0<\left(1+4\hat{C}_{3}T^{-1} +2T\|V\|_{\infty}+\|V\|^{2/3}_{\infty}\right)h_{0}<\hat{C}_{3}.
    \end{equation}
\item[($ii$)] There is a constant $C_{5}:= C_{5}(r,\delta)>C_{3}$ so that
\begin{equation}\label{1.10}
e^{2T\|V\|_{\infty}}\int_{T-\tau_{2}}^{T}\int_{B_{R}(x_{0})}|\varphi(x,s)|^{2}\mathrm{d}x\mathrm{d}s\leq e^{1+\frac{C_{5}}{(a+b^{2}/a)h_{0}}}\int_{B_{(1+\delta)r}(x_{0})}|\varphi(x,t)|^2\mathrm{d}x
\end{equation}
\end{description}
for each $t\in[T-\min\{\tau_{2},h_{0}\},T]$.
\end{Proposition}
\begin{proof}
For each $r'>0$, we write  $B_{r'}:= B_{r'}(0)$. Since $B_{2r}\subset B_{R}$ and
$$
e^{2C_{1}\left[1+(2a+\frac{|b|}{a})r^{-2}\right]\left[1+(\tau_{2}-\tau_{1})^{-1}+\|V\|^{2/3}_{\infty}\right]}
\geq C_{1} \left[(2a+\frac{|b|}{a})r^{-2}+(\tau_{2}-\tau_{1})^{-1}+\|V\|_{\infty}\right],
$$
by (\ref{1.2}) (where $R$ is replaced by $2r$), we have that
\begin{eqnarray*}
&&e^{2C_1\left[1+(2a+\frac{|b|}{a})r^{-2}\right]\left[1+(\tau_2-\tau_1)^{-1}+\|V\|^{2/3}_{\infty}\right]}
\displaystyle{\frac{\int_{T-\tau_2}^T\int_{B_R}|\varphi|^2\mathrm dx\mathrm dt}{\int_{B_r}|\varphi(x,T)|^2\mathrm dx}}\\
&\geq&C_1\left[(2a+\frac{|b|}{a})r^{-2}+(\tau_2-\tau_1)^{-1}+\|V\|_{\infty}\right]
\displaystyle{\frac{\int_{T-\tau_2}^T\int_{B_{2r}}|\varphi|^{2}\mathrm dx\mathrm dt}{\int_{B_r}|\varphi(x,T)|^2\mathrm dx}}\geq 1.
\end{eqnarray*}
Hence, (\ref{1.9}) follows immediately from (\ref{1.8}).

We now turn to the proof of (\ref{1.10}).
Let $h>0$, $\beta(x)=|x|^2$ and $\eta\in C_{0}^{\infty}(B_{(1+\delta)r})$
be such that
$$
0\leq\eta(\cdot)\leq 1\;\;\mbox{in}\;\;B_{(1+\delta)r}\;\;\mbox{and}\;\;
\eta(\cdot)=1\;\;\mbox{in}\;\;B_{(1+3\delta/4)r}.
$$
Multiplying the first equation of (\ref{1.1}) by $e^{-\beta/h}\eta^2\overline{\varphi}$ and
integrating it over $B_{(1+\delta)r}$, we get that
\begin{equation}\label{1.99999}
\begin{array}{lll}
&&\displaystyle{}\int_{B_{(1+\delta)r}}e^{-\beta/h}\eta^{2}\varphi_{t}\overline{\varphi}\mathrm{d}x
+(a+ib)\int_{B_{(1+\delta)r}}\nabla\varphi\cdot\nabla(e^{-\beta/h}\eta^2\overline{\varphi})\mathrm{d}x\\
\\
&=&\displaystyle{}-\int_{B_{(1+\delta)r}}Ve^{-\beta/h}\eta^{2}\varphi\overline{\varphi}\mathrm{d}x.
\end{array}
\end{equation}
By (\ref{1.99999}), we have that
\begin{equation}\label{I1}
\begin{array}{lll}
&&\displaystyle{}\int_{B_{(1+\delta)r}}e^{-\beta/h}\eta^{2}\overline{\varphi}_{t}\varphi\mathrm{d}x
+(a-ib)\int_{B_{(1+\delta)r}}\nabla\overline{\varphi}\cdot\nabla(e^{-\beta/h}\eta^2\varphi)\mathrm{d}x\\
\\
&=&\displaystyle{}-\int_{B_{(1+\delta)r}}\overline{V}e^{-\beta/h}\eta^{2}\varphi\overline{\varphi}\mathrm{d}x.
\end{array}
\end{equation}
Since
\begin{equation}\label{I2}
\nabla(e^{-\beta/h}\eta^{2}\overline{\varphi})
=-\frac{1}{h}e^{-\beta/h}\eta^{2}\overline{\varphi}\nabla\beta+2e^{-\beta/h}\eta\overline{\varphi}\nabla\eta+e^{-\beta/h}\eta^{2}\nabla\overline{\varphi},
\end{equation}
by (\ref{1.99999}), (\ref{I1}) and (\ref{I2}), we have that
\begin{eqnarray*}
&&\frac{\mathrm{d}}{\mathrm{d}t}\int_{B_{(1+\delta)r}}e^{-\beta/h}|\eta\varphi|^{2}\mathrm{d}x
+2a\int_{B_{(1+\delta)r}}e^{-\beta/h}|\eta\nabla\varphi|^{2}\mathrm{d}x\\
&=&2a\int_{B_{(1+\delta)r}}\frac{1}{h}e^{-\beta/h}\eta^{2}\left(\varphi_{1}\nabla\varphi_{1}\nabla\beta+\varphi_{2}\nabla\varphi_{2}\nabla\beta\right)
-2e^{-\beta/h}\eta\left(\varphi_{1}\nabla\varphi_{1}\nabla\eta+\varphi_{2}\nabla\varphi_{2}\nabla\eta\right)\mathrm{d}x\\
&&-2b\int_{B_{(1+\delta)r}}\frac{1}{h}e^{-\beta/h}\eta^{2}\left(\varphi_{1}\nabla\varphi_{2}\nabla\beta-\varphi_{2}\nabla\varphi_{1}\nabla\beta\right)
-2e^{-\beta/h}\eta\left(\varphi_{1}\nabla\varphi_{2}\nabla\eta-\varphi_{2}\nabla\varphi_{1}\nabla\eta\right)\mathrm{d}x\\
&&-2\int_{B_{(1+\delta)r}}(V+\overline{V})e^{-\beta/h}|\eta\varphi|^{2}\mathrm{d}x.
\end{eqnarray*}
Thus, we have that
\begin{eqnarray*}
&&\frac{\mathrm{d}}{\mathrm{d}t}\int_{B_{(1+\delta)r}}e^{-\beta/h}|\eta\varphi|^{2}\mathrm{d}x
+2a\int_{B_{(1+\delta)r}}e^{-\beta/h}|\eta\nabla\varphi|^{2}\mathrm{d}x\\
&\leq&2a\int_{B_{(1+\delta)r}}e^{-\beta/h}|\eta\nabla\varphi|^{2}\mathrm{d}x
+\frac{2}{h^{2}}\left(a+\frac{b^2}{a}\right)\int_{B_{(1+\delta)r}}e^{-\beta/h}\eta^{2}|\varphi|^{2}|\nabla\beta|^{2}\mathrm{d}x
\\
&&
+4\left(2a+\frac{b^2}{a}\right)\int_{B_{(1+\delta)r}}e^{-\beta/h}|\varphi|^{2}|\nabla\eta|^{2}\mathrm{d}x
+2\|V\|_{\infty}\int_{B_{(1+\delta)r}}e^{-\beta/h}|\eta\varphi|^{2}\mathrm{d}x.
\end{eqnarray*}
This, along with  Cauchy-Schwarz inequality, implies that
\begin{eqnarray*}
\frac{\mathrm{d}}{\mathrm{d}t}\int_{B_{(1+\delta)r}}e^{-\beta/h}|\eta\varphi|^{2}\mathrm{d}x
&\leq&\left[\left(a+\frac{b^{2}}{a}\right)\frac{8(1+\delta)^2 r^{2}}{h^2}+2\|V\|_{\infty}\right]\int_{B_{(1+\delta)r}}e^{-\beta/h}|\eta\varphi|^{2}\mathrm{d}x\\
&&+4\left(2a+\frac{b^{2}}{a}\right)\int_{\{x:(1+3\delta/4)r\leq\sqrt{\beta(x)}\leq(1+\delta)r\}}|\nabla\eta|^{2}e^{-\beta/h}|\varphi|^{2}\mathrm{d}x,
\end{eqnarray*}
which indicates that
\begin{eqnarray*}
\frac{\mathrm{d}}{\mathrm{d}t}\int_{B_{(1+\delta)r}}e^{-\beta/h}|\eta\varphi|^{2}\mathrm{d}x
&\leq&\left[\left(a+\frac{b^{2}}{a}\right)\frac{8(1+\delta)^2 r^{2}}{h^2}+2\|V\|_{\infty}\right]\int_{B_{(1+\delta)r}}e^{-\beta/h}|\eta\varphi|^{2}\mathrm{d}x\\
&&+4\left(2a+\frac{b^{2}}{a}\right)\|\nabla\eta\|^{2}_{\infty}e^{-\frac{(1+3\delta/4)^2 r^2}{h}}\int_{B_{(1+\delta)r}}|\varphi|^{2}\mathrm{d}x.
\end{eqnarray*}
Here and throughout the proof of Proposition~\ref{lemma-1.3},
$\|\nabla\eta\|_{\infty}:=\|\nabla\eta\|_{L^{\infty}(B_{(1+\delta)r})}$.
From the latter it follows that
\begin{eqnarray*}
&&\frac{\mathrm{d}}{\mathrm{d}t}\left\{e^{-\left[\left(a+\frac{b^{2}}{a}\right)\frac{8(1+\delta)^2 r^{2}}{h^2}+2\|V\|_{\infty}\right]t}
\int_{B_{(1+\delta)r}}e^{-\beta/h}|\eta\varphi|^{2}\mathrm{d}x\right\}\\
&\leq&4\left(2a+\frac{b^{2}}{a}\right)\|\nabla\eta\|^{2}_{\infty}e^{-\left[\left(a+\frac{b^{2}}{a}\right)\frac{8(1+\delta)^2 r^{2}}{h^2}+2\|V\|_{\infty}\right]t}
e^{-\frac{(1+3\delta/4)^2 r^2}{h}}\int_{B_{(1+\delta)r}}|\varphi|^{2}\mathrm{d}x.
\end{eqnarray*}
Integrating the latter inequality over $(t,T)$, we get that
\begin{equation}\label{1.109999}
\begin{array}{lll}
&&\displaystyle{\int_{B_{(1+\delta)r}}}e^{-\beta/h}|\eta\varphi(x,T)|^{2}\mathrm dx\\
&\leq& e^{\left[\left(a+\frac{b^{2}}{a}\right)\frac{8(1+\delta)^2 r^{2}}{h^2}+2\|V\|_{\infty}\right](T-t)}
\displaystyle{\int_{B_{(1+\delta)r}}}e^{-\beta/h}|\eta\varphi(x,t)|^{2}\mathrm{d}x\\
&&+4\left(2a+\frac{b^{2}}{a}\right)e^{\left[\left(a+\frac{b^{2}}{a}\right)\frac{8(1+\delta)^2 r^{2}}{h^2}+2\|V\|_{\infty}\right](T-t)}\|\nabla\eta\|^2_{\infty}
e^{-\frac{(1+3\delta/4)^2 r^2}{h}}\displaystyle{\int_t^T\int_{B_{(1+\delta)r}}}|\varphi|^{2}\mathrm{d}x\mathrm{d}s.
\end{array}
\end{equation}
 We simply write  $b_{1}:= 8(1+\delta)^{2}, b_{2}:= (1+3\delta/4)^{2}$
 and $b_{3}:= (1+\delta/2)^{2}.$ It is clear that $1<b_{3}<b_{2}$.
 Recall that $t\leq T$. We now suppose that a positive real number $h$ is such that
 $$
 0<T-\frac{(b_{2}-b_{3})h}{(a+b^{2}/a)b_{1}}\leq t.
 $$
 Then $b_{1}(a+b^{2}/a)(T-t)/h^{2}\leq(b_{2}-b_{3})/h$ and (\ref{1.109999}) yields
\begin{eqnarray*}
\int_{B_{(1+\delta)r}}e^{-\beta/h}|\eta\varphi(x,T)|^{2}\mathrm{d}x
&\leq& e^{\frac{(b_{2}-b_{3})r^{2}}{h}}e^{2T\|V\|_{\infty}}
\int_{B_{(1+\delta)r}}e^{-\beta/h}|\eta\varphi(x,t)|^{2}\mathrm{d}x\\
&&+4\left(2a+\frac{b^{2}}{a}\right)\|\nabla\eta\|^{2}_{\infty}e^{2T\|V\|_{\infty}}e^{\frac{-b_{3}r^{2}}{h}}
\int_{t}^{T}\int_{B_{(1+\delta)r}}|\varphi|^{2}\mathrm{d}x\mathrm{d}s.
\end{eqnarray*}
Since $\eta(\cdot)=1$ in $B_{r}$, the above estimate gives
\begin{equation}\label{1.15}
\begin{array}{lll}
\displaystyle{}\int_{B_{r}}|\varphi(x,T)|^{2}\mathrm{d}x&\leq& e^{\frac{(b_{2}-b_{3}+1)r^{2}}{h}}e^{2T\|V\|_{\infty}}
\displaystyle{\int_{B_{(1+\delta)r}}}e^{-\beta/h}|\eta\varphi(x,t)|^{2}\mathrm{d}x\\
&&+8\left(a+\frac{b^{2}}{a}\right)\|\nabla\eta\|^{2}_{\infty}e^{2T\|V\|_{\infty}}e^{\frac{-(b_{3}-1)r^{2}}{h}}
\displaystyle{\int_{t}^{T}\int_{B_{(1+\delta)r}}}|\varphi|^{2}\mathrm{d}x\mathrm{d}s
\end{array}
\end{equation}
whenever $0<T-(b_{2}-b_{3})h/[(a+b^{2}/a)b_{1}]\leq t\leq T$. Recall that $h_{0}<T$ from \eqref{1.9}. We choose $h$ as follows:
$$
h=\frac{(a+b^{2}/a)b_{1}}{b_{2}-b_{3}}h_{0}=\frac{b_{1}C_{3}/(b_{2}-b_{3})}
{\ln\left[(1+C_{4})\left(e^{\left[1+2C_{1}(1+\frac{1}{r^2})\right]
(1+\frac{1}{\tau_{2}-\tau_{1}}+\|V\|^{2/3}_{\infty})+\frac{4C_{3}}{T}+2T\|V\|_{\infty}}\right)
\frac{\int_{T-\tau_{2}}^{T}\int_{B_{R}}|\varphi|^{2}\mathrm{d}x\mathrm{d}t}
{\int_{B_{r}}|\varphi(x,T)|^{2}\mathrm{d}x}\right]}
$$
with $C_{3}:=(b_{2}-b_{3})(b_{3}-1)r^{2}/b_{1}$ and
$C_{4}:=8\|\nabla\eta\|^{2}_{\infty}$.
Then for any $t\in[T-\min\{\tau_{2},h_{0}\},T]$, we have that
\begin{equation}\label{2.211111}
\begin{array}{lll}
&&\displaystyle{}8\left(a+\frac{b^{2}}{a}\right)\|\nabla\eta\|^{2}_{\infty}e^{2T\|V\|_{\infty}}
e^{-\frac{(b_{3}-1)r^{2}}{h}}\displaystyle{}\int_{t}^{T}\int_{B_{(1+\delta)r}}|\varphi|^{2}\mathrm{d}x\mathrm{d}s\\
\\
&=&\displaystyle{}\frac{C_{4}(a+b^{2}/a)e^{2T\|V\|_{\infty}}\int_{t}^{T}\int_{B_{(1+\delta)r}}|\varphi|^{2}\mathrm{d}x\mathrm{d}s}
{[1+(a+b^{2}/a)C_{4}]\left(e^{\left[1+2C_{1}(1+\frac{1}{r^2})\right]
(1+\frac{1}{\tau_{2}-\tau_{1}}+\|V\|^{2/3}_{\infty})+\frac{4C_{3}}{T}+2T\|V\|_{\infty}}\right)
\frac{\int_{T-\tau_{2}}^{T}\int_{B_{R}}|\varphi|^{2}\mathrm{d}x\mathrm{d}s}{\int_{B_{r}}|\varphi(x,T)|^{2}\mathrm{d}x}}\\
\\
&\leq&\displaystyle{} \frac{1}{e}\displaystyle{}\int_{B_{r}}|\varphi(x,T)|^{2}\mathrm{d}x.
\end{array}
\end{equation}
(In the latter inequality, we used the facts that $(1+\delta)r\leq2r\leq R.$)

Next, on one hand, by \eqref{1.15} and \eqref{2.211111}, we get that
\begin{equation}\label{2.22222}
\left(1-\frac{1}{e}\right)\int_{B_{r}}|\varphi(x,T)|^{2}\mathrm{d}x\leq
e^{\frac{(b_{2}-b_{3}+1)(b_{2}-b_{3})r^{2}}{(a+b^{2}/a)b_{1}h_{0}}}
e^{2T\|V\|_{\infty}}\int_{B_{(1+\delta)r}}|\varphi(x,t)|^{2}\mathrm{d}x
\end{equation}
for each $T-\min{\{\tau_{2},h_{0}\}}\leq t\leq T.$
On the other hand, by \eqref{1.8}, we have that
\begin{equation*}
\frac{\int_{T-\tau_{2}}^{T}\int_{B_{R}}|\varphi|^{2}\mathrm{d}x\mathrm{d}s}
{\int_{B_{r}}|\varphi(x,T)|^{2}\mathrm{d}x}\leq e^{\frac{C_{3}}{(a+b^{2}/a)h_{0}}},
\end{equation*}
which, combined with \eqref{2.22222}, indicates that
$$
\left(1-\frac{1}{e}\right)e^{-\frac{C_{3}}{(a+b^{2}/a)h_{0}}}
\int_{T-\tau_{2}}^{T}\int_{B_{R}}|\varphi|^{2}\mathrm{d}x\mathrm{d}s\leq
e^{\frac{(b_{2}-b_{3}+1)(b_{2}-b_{3})r^{2}}{(a+b^{2}/a)b_{1}h_{0}}}
e^{2T\|V\|_{\infty}}\int_{B_{(1+\delta)r}}|\varphi(x,t)|^{2}\mathrm{d}x
$$
for each $T-\min{\{\tau_{2},h_{0}\}}\leq t\leq T.$
Since $2T\|V\|_{\infty}h_{0}<C_{3}/(a+b^{2}/a)$ (see \eqref{1.9}), the desired estimate \eqref{1.10}
follows from the latter inequality immediately with $C_{5}:=3C_{3}+(b_{2}-b_{3}+1)(b_{2}-b_{3})r^{2}/b_{1}$.
\end{proof}

\section{Logarithmic convexity}\label{pre}

In the proof of Theorem \ref{Thm1} in this paper, the three lemmas below related to the log-convexity method play a crucial role. Specifically, they are employed in key steps to control the growth of the solution, to establish the required uniform a priori estimates, and to derive the differential inequalities satisfied by the energy.

\begin{lemma}\label{lemma4}(\cite{Bardos-Phung})
Let $\hbar>0, T>0$ and $F_1, F_2 \in C([0, T])$. Consider two positive functions $y, N \in C^1([0, T])$ such that
$$
\left\{\begin{array}{l}
\left|\frac{1}{2} y^{\prime}(t)+N(t) y(t)\right| \leq\left(\frac{1}{2} N(t)+\frac{C_0}{T-t+\hbar}\right) y(t)+F_1(t) y(t) \\
N^{\prime}(t) \leq\left(\frac{1+C_0}{T-t+\hbar}\right) N(t)+F_2(t)
\end{array}\right.
$$
with $C_0> 0$. Then for any  $\ell>1$ such that $\ell \hbar<\min (1/2, T/4)$, one has
$$
y(T-\ell \hbar)^{1+M_{\ell}} \leq y(T) y(T-2 \ell \hbar)^{M_{\ell}} \mathrm{e}^{D_{\ell}}(2 \ell+1)^{3 C_0\left(1+M_{\ell}\right)}
$$
where $$D_{\ell}=3\left(1+M_{\ell}\right)\int_{T-2 \ell \hbar}^{T}(\left|F_1\right|+2\ell\hbar\left|F_2\right|) \mathrm{d} t \ and \  M_{\ell} \leq 3 \frac{(\ell+1)^{C_0}}{1-\left(\frac{2}{3}\right)^{C_0}}.$$
\end{lemma}

\begin{lemma}\label{lemma-2.5}(\cite{luis})
Assume that $\mathcal S$ is a symmetric operator, and that $\mathcal A$ is skew-symmetric, both are allowed to depend on the time variable. Suppose $z(x, t)$ is a reasonable function.
%$z\in C_0^\infty(\mathbb R^N)$.
Then
$$
\frac{dN(t)}{dt}\leq \frac{\langle-(\partial_t\mathcal S+[\mathcal S,\mathcal A])z,z\rangle}{\|z\|^2}
+\frac{\|\partial_t z-\mathcal Sz-\mathcal Az\|^2}{2\|z\|^2}$$
where $N(t)=-\langle \mathcal Sz, z\rangle/\|z\|^2$ and $[\mathcal S,\mathcal A]$ is the commutator of $\mathcal S$ and $\mathcal A$.
\end{lemma}
For each $r'>0$, we denote $B_{r'}:= B_{r'}(0)$.  Let $\chi\in C_{0}^{\infty}(B_{R_{0}})$ be such that
$$
0\leq\chi(\cdot)\leq 1 \ \mathrm{in} \ B_{R_{0}} \ \mathrm{and} \ \chi(\cdot)=1 \ \mathrm{in} \ B_{(1+3\delta/2)R}.$$
We set $u:=\chi\varphi$. It is clear that

\begin{equation}\label{3.222}
\left\{\begin{array}{lll}
\partial_t u-(a+ib)\Delta u+V(x,t)u=g&\mathrm{in}\ B_{R_{0}}\times(0,T),\\
u=0&\mathrm{on}\ \partial B_{R_{0}}\times(0,T),\\
u(x,0)=\chi\varphi_0&\mathrm{in}\ B_{R_{0}}.
\end{array}\right.
\end{equation}
Furthermore, we define  $g:=-(a+ib)[2\nabla\chi\cdot\nabla\varphi+\varphi\Delta\chi].$

For each $\lambda>0$, throughout this paper we define $\Gamma(t)=T-t+\lambda,\;t\in[0,T].$
   Given a constant $\gamma_{0}>0$ (which will be fixed later), we let
\begin{equation*}
y(x,t)=-\frac{\gamma_{0}|x|^2}{\Gamma(t)}, \quad (x,t)\in B_{R_{0}}\times[0,T].
\end{equation*}
Then it can be easily checked that
\begin{equation}\label{2.35a}
\begin{split}
&\partial_t y=-\frac{\gamma_{0}|x|^2}{\Gamma^2(t)},
\partial_t^2 y=-\frac{2\gamma_{0}|x|^2}{\Gamma^3(t)},
\nabla y=-\frac{2\gamma_{0}x}{\Gamma(t)},
\Delta^2 y=0,\\
&\nabla\partial_ty=-\frac{2\gamma_{0} x}{\Gamma^2(t)},
\Delta y=-\frac{2n\gamma_{0}}{\Gamma(t)},
\Delta\partial_t y=-\frac{2n\gamma_{0}}{\Gamma^2(t)},
D^2 y =-\frac{2\gamma_{0}}{\Gamma(t)}\mathbb I_{n\times n}.
\end{split}
\end{equation}
Define
\begin{equation}\label{2.35b}
\mathcal S:=a\Delta-\frac{\gamma_{0}/2-a\gamma_{0}^2}{\Gamma^2(t)}|x|^2
+i\frac{b\gamma_{0}}{\Gamma(t)}\left[2x\cdot\nabla+n\right],
\end{equation}

\begin{equation}\label{sz.35b}
\mathcal A := ib\left(\Delta+\frac{\gamma^{2}_{0}}{\Gamma^{2}(t)}|x|^2\right)+ \frac{a\gamma_{0}}{\Gamma(t)}\left[2x\cdot\nabla+n\right].
\end{equation}
By \eqref{2.35a}, \eqref{2.35b} and \eqref{sz.35b}, we have that
\begin{equation}\label{2.36b}
\begin{split}
\partial_t\mathcal S+[\mathcal S,\mathcal A]
=&\frac{4\gamma_{0}(a^2+b^2)}{\Gamma(t)}\Delta-\frac{\gamma_{0}-4a\gamma_{0}^2+4(a^2+b^2)\gamma_{0}^3}{\Gamma^3(t)}|x|^2\\
&+i\frac{2b\gamma_{0}}{\Gamma^2(t)}[2x\cdot\nabla+n].
\end{split}
\end{equation}
Setting
$f=e^{\frac y 2}u$, we have (see formulas (2.12), (2.13) and (2.14) in
\cite[Lemma 3]{luis} with $\gamma=1/2$)
\begin{equation}\label{2.29a}
\partial_t f-\mathcal Sf -\mathcal Af=-Vf+e^{\frac{y}{2}}g.
\end{equation}
One can also check that for any $f_{1},f_{2}\in H^1_0(B_{R_{0}};\mathbb C)$
\begin{equation}\label{2.31C}
\langle \mathcal Af_{1},f_{2}\rangle=-\langle f_{1},\mathcal A f_{2}\rangle\
\mathrm{and}\ \langle \mathcal Sf_{1},f_{2}\rangle=\langle f_{1},\mathcal S f_{2}\rangle.
\end{equation}

Given a constant $\varepsilon_{0}\in(0,1)$ (which will be fixed later), by \eqref{2.35b} and \eqref{2.36b}, we then concern the operator
\begin{equation}\label{2.37a}
\begin{split}
&-\left(\partial_t\mathcal S+[\mathcal S,\mathcal A]\right)+(2-\varepsilon_{0})\frac{\mathcal S}{\Gamma(t)}\\
=&\frac{(2-\varepsilon_{0})a-4\gamma_{0}(a^2+b^2)}{\Gamma(t)}\Delta
+\frac{\gamma_{0}\left[1-4a\gamma_{0}
+4(a^2+b^2)\gamma_{0}^2-(2-\varepsilon_{0})(1/2-a\gamma_{0})\right]}{\Gamma^3(t)}|x|^2\\
&-i\frac{\varepsilon_{0} b\gamma_{0}}{\Gamma^2(t)}[2x\cdot\nabla+n].
\end{split}
\end{equation}
\begin{lemma}\label{lemma-4r4}
Under the above assumptions, we have
\begin{equation}\label{Y2}
\frac 12\frac{d}{dt}\|f(t)\|^2+\langle -\mathcal Sf(t),f(t)\rangle=-\text{Re}(Vf(t),f(t))+\frac{1}{2}\int_{B_{R_{0}}}e^{\frac{y}{2}}[g\overline{f}+f\overline{g}]\mathrm{d}x
\end{equation}
and
\begin{equation}\label{Y1}
\frac{d}{dt}N(t)\leq\frac{2-\varepsilon_{0}}{\Gamma(t)}N(t)+\frac{\|-Vf+e^{\frac{y}{2}}g\|^{2}}{2\|f\|^{2}}
\end{equation}
where
\begin{equation*}\label{gutou2}
N(t):=\frac{\langle -\mathcal Sf(t),f(t)\rangle}{\|f(t)\|^2},\quad t\in[0,T].
\end{equation*}
\end{lemma}
\begin{proof}
By \eqref{2.31C}, we have  $\mathcal S$ and $\mathcal A$ are symmetric and skew-symmetric, respectively. Then, along with \eqref{2.29a}, implies \eqref{Y2}. We now turn to the proof of \eqref{Y1}. By the Cauchy-Schwartz inequality, it holds  that
\begin{equation}\label{2.40a}
\gamma_{0}\,\big|\text{Im}\langle 2x\cdot\nabla f,f\rangle\big|\leq \Gamma(t)\|\nabla f\|^2+\frac{\gamma_{0}^2}{\Gamma(t)}
\big\||x|f\big\|^2.
\end{equation}
Hence, by \eqref{2.40a} and \eqref{2.37a}, we have
\begin{equation}\label{2.41a}
\begin{split}
&\left\langle\left\{-\left(\partial_t\mathcal S+[\mathcal S,\mathcal A]\right)+(2-\varepsilon_{0})\frac{\mathcal S}{\Gamma(t)}\right\}f,f\right\rangle\\
\leq& -\frac{(2-\varepsilon_{0})a-4\gamma_{0}(a^2+b^2)-\varepsilon_{0} |b|}{\Gamma(t)}\|\nabla f\|^2\\
&+\frac{\gamma_{0}\left\{
\varepsilon_{0}/2-[2a+(a-|b|)\varepsilon_{0}]\gamma_{0}+4(a^2+b^2)\gamma_{0}^2\right\}}{\Gamma^3(t)}
\big\||x|f\big\|^2.
\end{split}
\end{equation}
Denote by
\begin{equation*}
\text{I}:= (2-\varepsilon_{0})a-4\gamma_{0}(a^2+b^2)-\varepsilon_{0} |b|,
\end{equation*}
\begin{equation*}
\text{II}:= \varepsilon_{0}/2-[2a+(a-|b|)\varepsilon_{0}]\gamma_{0}+4(a^2+b^2)\gamma_{0}^2.
\end{equation*}
Now, we first fix the above parameter $\varepsilon_{0}\in(0,1)$ to be such that
\begin{equation*}
[2a+(a-|b|)\varepsilon_{0}]^2-8\varepsilon_{0}(a^2+b^2)\geq0
\end{equation*}
and
\begin{equation*}
0<\frac{2a+(a-|b|)\varepsilon_{0}}{8(a^2+b^2)}\leq \frac{2a-(a+|b|)\varepsilon_{0}}{4(a^2+b^2)};
\end{equation*}
and then we choose
\begin{equation*}
\gamma_{0}=\frac{2a+(a-|b|)\varepsilon_{0}}{8(a^2+b^2)}.
\end{equation*}
Based on the above definitions, one can easily check that $\text{I}\geq0$ and $\text{II}\leq 0$. Thus, we have
\begin{equation}\label{gutou1}
\left\langle\left\{-\left(\partial_t\mathcal S+[\mathcal S,\mathcal A]\right)+(2-\varepsilon_{0})\frac{\mathcal S}{\Gamma(t)}\right\}f,f\right\rangle\leq 0.
\end{equation}
By Lemma~\ref{lemma-2.5} , \eqref{2.29a} and \eqref{gutou1} we obtain
\begin{equation*}
\frac{d}{dt}N(t)\leq\frac{2-\varepsilon_{0}}{\Gamma(t)}N(t)+\frac{\|-Vf+e^{\frac{y}{2}}g\|^{2}}{2\|f\|^{2}}.
\end{equation*}

In summary, we finish the proof of this lemma.
\end{proof}

\begin{remark}
 We note the limiting behavior for the choices of \( \varepsilon_0 \) and \( \gamma_0 \), namely, when \( a \approx 0 \) and \( b = 1 \), one may take \( \varepsilon_0 = a^2 / 8 \) and \( \gamma_0 \approx a / 4 \), where \( \approx \) denotes ``approximately equal to''.
\end{remark}

\section{Proof\ of\ Theorem~\ref{Thm1}}\label{pr3}

%Prior to proving Theorem~\ref{Thm1}, it is necessary to establish several energy estimates for equation \eqref{1.1}, as presented in Lemmas~\ref{lemma-1.1}, ~\ref{lemma-1.2}, and Proposition~\ref{lemma-1.3}. While the proofs of these lemmas follow arguments analogous to those in Reference \cite{Duan-Wang-Zhang}, we provide full proofs in the Appendix.\\
\noindent\textbf{Proof\ of\ Theorem~\ref{Thm1}}.
For each $r'>0$, we denote $B_{r'}:= B_{r'}(x_{0})$.  Let $\chi\in C_{0}^{\infty}(B_{R_{0}})$ be such that
$$
0\leq\chi(\cdot)\leq 1 \ \mathrm{in} \ B_{R_{0}} \ \mathrm{and} \ \chi(\cdot)=1 \ \mathrm{in} \ B_{(1+3\delta/2)R}.$$
Note that $g$ is supported on $\{x: (1+3\delta/2)R\leq|x|\leq R_{0}\}.$
Recall that $\chi(\cdot)=1$ in $B_{(1+\delta)R}$. We can easily check that
\begin{eqnarray}\label{3.333}
\displaystyle{}\displaystyle{}\frac{\int_{B_{R_{0}}}e^{\frac{y}{2}}[g\overline{f}+f\overline{g}]\mathrm{d}x}
{\int_{B_{R_{0}}}|f|^{2}\mathrm{d}x}
\displaystyle{}&=&\displaystyle{}\frac{-2a\int_{B_{R_{0}}\setminus B_{(1+3\delta/2)R}}
e^{y}\left[\Delta\chi\chi \varphi\overline{\varphi}+2\chi\nabla\chi(\nabla \varphi_{1}\varphi_{1}+\nabla \varphi_{2}\varphi_{2})\right]\mathrm{d}x}
{\int_{B_{R_{0}}}|\chi\varphi(x,t)|^{2}e^{y}\mathrm{d}x}\nonumber\\
\displaystyle{}&&+\displaystyle{}\frac{4b\int_{B_{R_{0}}\setminus B_{(1+3\delta/2)R}}
e^{y}\chi\nabla\chi(\nabla \varphi_{2}\varphi_{1}+\nabla \varphi_{1}\varphi_{2})\mathrm{d}x}
{\int_{B_{R_{0}}}|\chi\varphi(x,t)|^{2}e^{y}\mathrm{d}x}\\
\displaystyle{}&\leq&\displaystyle{}  e^{-\frac{\mathcal{K}_{1}}{T-t+\lambda}}\frac{2a\int_{B_{R_{0}}\setminus B_{(1+3\delta/2)R}}\left(2|\nabla\chi\nabla \varphi_{1}\varphi_{1}+\nabla\chi\nabla \varphi_{2}\varphi_{2}|
+|\Delta\chi||\varphi|^{2}\right)\mathrm{d}x}{\int_{B_{(1+\delta)R}}|\varphi|^{2}\mathrm{d}x}\nonumber\\
\displaystyle{}&&+\displaystyle{} e^{-\frac{\mathcal{K}_{1}}{T-t+\lambda}}\frac{4|b|\int_{B_{R_{0}}\setminus B_{(1+3\delta/2)R}}|\nabla\chi\nabla \varphi_{2}\varphi_{1}+\nabla\chi\nabla \varphi_{1}\varphi_{2}|\mathrm{d}x}{\int_{B_{(1+\delta)R}}|\varphi|^{2}\mathrm{d}x}\nonumber
\end{eqnarray}
where $\mathcal{K}_{1}:=[(1+3\delta/2)^2-(1+\delta)^2]\gamma_{0}R^2/4$. It follows from  \eqref{3.333} that
\begin{equation}\label{3.444}
\begin{array}{lll}
\displaystyle{}\frac{\int_{B_{R_{0}}}e^{\frac{y}{2}}[g\overline{f}+f\overline{g}]\mathrm{d}x}
{\int_{B_{R_{0}}}|f|^{2}\mathrm{d}x}
\displaystyle{}&\leq& \displaystyle{} e^{-\frac{\mathcal{K}_{1}}{T-t+\lambda}}\frac{4(a+|b|)\|\nabla\chi\|_{\infty}
\left(\int_{B_{R_{0}}}|\varphi|^{2}\mathrm{d}x\right)^{\frac{1}{2}}\left(\int_{B_{R_{0}}}
|\nabla\varphi|^{2}\mathrm{d}x\right)^{\frac{1}{2}}}
{\int_{B_{(1+\delta)R}}|\varphi|^{2}\mathrm{d}x}\\
\\
\displaystyle{}&&+ \displaystyle{} e^{-\frac{\mathcal{K}_{1}}{T-t+\lambda}}\frac{2a\|\Delta\chi\|_{\infty}\int_{B_{R_{0}}}|\varphi|^{2}\mathrm{d}x}
{\int_{B_{(1+\delta)R}}|\varphi|^{2}\mathrm{d}x}.
\end{array}
\end{equation}
Here $\|\nabla\chi\|_{\infty}:=\|\nabla\chi\|_{L^{\infty}(B_{R_{0}})}$ and $\|\Delta\chi\|_{\infty}:=\|\Delta\chi\|_{L^{\infty}(B_{R_{0}})}.$

To achieve the goal, we divide the proof process into four steps.

\textbf{Step 1.} We claim that
\begin{equation}\label{3.1yy}
\left\{\begin{array}{lll}
\displaystyle{}\frac 12\frac{d}{dt}\|f(t)\|^2+N(t)\|f(t)\|^2
\leq\displaystyle{}\left[e^{-\frac{\mathcal{K}_{1}}{T-t+\lambda}}e^{1+\frac{C_{5}}{(a+b^{2}/a)h_{0}}}
\left(1+\frac{|b|}{a}\right)
\left(a^2+\frac{b^2}{a}+T^{-2}\right)+\|V\|_{\infty}\right]\|f(t)\|^2,\\
\\
\displaystyle{}\frac{d}{dt}N(t)\leq\frac{2-\varepsilon_{0}}{\Gamma(t)}N(t)+(a^2+b^2)\left(a+\frac{|b|}{a}+\frac{b^4}{a^3}+T^{-2}\right)
e^{1+\frac{C_{5}}{(a+b^{2}/a)h_{0}}}e^{-\frac{\mathcal{K}_{1}}{T-t+\lambda}} +\|V\|^{2}_{\infty}.
\end{array}\right.
\end{equation}

On one hand, by Lemma~\ref{lemma-1.1} (where $r, R,\tau_{1} $ and $\tau_{2}$ are replaced by
$R_{0}, 2R_{0}, T/4$ and $T/2$, respectively), we have that
\begin{equation}\label{3.555}
\int_{B_{R_{0}}}|\varphi(x,t)|^{2}\mathrm{d}x
\leq C\left(a+\frac{|b|}{a}+T^{-1}+\|V\|_{\infty}\right)\int_{T/2}^{T}\int_{B_{2R_{0}}}|\varphi|^{2}
\mathrm{d}x\mathrm{d}s \  \ \mathrm{for \ each} \ t\in [3T/4,T],
 \end{equation}
where $C:=C(R)>0.$
By Lemma~\ref{lemma-1.2} (where $ R $ and $\tau$ are replaced by  $R_{0}$ and $T/4$, respectively),
we get that
\begin{equation}\label{3.666}
\int_{B_{R_{0}}}|\nabla\varphi(x,t)|^{2}\mathrm{d}x\leq C\left(a+\frac{b^{4}}{a^{3}}+a^{-1}T^{-2}+a^{-1}\|V\|^{2}_{\infty}\right)\int_{T/2}^{T}\int_{B_{2R_{0}}}|\varphi|^{2}\mathrm{d}x\mathrm{d}s
\  \ \mathrm{for \ each} \ t\in [3T/4,T],
\end{equation}
 where $C:=C(R)>0.$
 It follows from \eqref{3.444}-\eqref{3.666} that for each $ t\in [3T/4,T]$,
\begin{equation}\label{3.777}
\displaystyle{}\frac{\int_{B_{R_{0}}}e^{\frac{y}{2}}[g\overline{f}+f\overline{g}]\mathrm{d}x}
{\int_{B_{R_{0}}}|f|^{2}\mathrm{d}x}
\leq\displaystyle{} e^{-\frac{\mathcal{K}_{1}}{T-t+\lambda}}
\frac{(1+|b|/a)(a^{2}+b^{2}/a+T^{-2}+a^{1/2}\|V\|^{3/2}_{\infty})
\int_{T/2}^{T}\int_{B_{2R_{0}}}|\varphi|^{2}\mathrm{d}x\mathrm{d}s}
{\int_{B_{(1+\delta)R}}|\varphi(x,t)|^{2}\mathrm{d}x}.
\end{equation}
 According to \eqref{1.10} (where $r, R,\tau_{1} $ and $\tau_{2}$ are replaced by
 $R, 2R_{0}, T/4$ and $T/2$, respectively), it holds that
 \begin{equation}\label{3.888}
\begin{array}{lll}
e^{2T\|V\|_{\infty}}\int_{T/2}^{T}\int_{B_{2R_{0}}}|\varphi(x,t)|^{2}\mathrm{d}x\mathrm{d}s
&\leq& e^{1+\frac{C_{5}}{(a+b^{2}/a)h_{0}}}
\int_{B_{(1+\delta)R}}|\varphi(x,t)|^{2}\mathrm{d}x
\end{array}
\end{equation}
for each $t\in [T-h_{0},T]$. Here, we used the fact that $h_{0}<T/4$ (see \eqref{1.9}). This, along with \eqref{3.777}, implies that
\begin{equation}\label{3.999}
\begin{array}{lll}
\displaystyle{}\frac{\int_{B_{R_{0}}}e^{\frac{y}{2}}[g\overline{f}+f\overline{g}]\mathrm{d}x}
{\int_{B_{R_{0}}}|f|^{2}\mathrm{d}x}\leq \displaystyle{} e^{-\frac{\mathcal{K}_{1}}{T-t+\lambda}}e^{1+\frac{C_{5}}{(a+b^{2}/a)h_{0}}}
\left(1+\frac{|b|}{a}\right)
\left(a^2+\frac{b^2}{a}+T^{-2}\right)
 \ \ \  \mathrm{for \ each}  \ \ t\in [T-h_{0},T].
\end{array}
\end{equation}
On the other hand, by similar arguments as those for \eqref{3.444}, we have that
 \begin{eqnarray*}
\displaystyle{}\frac{\int_{B_{R_{0}}} e^{y}|g|^{2}\mathrm{d}x}{\int_{B_{R_{0}}}|f|^{2}\mathrm{d}x}
&\leq& e^{-\frac{\mathcal{K}_{1}}{T-t+\lambda}}(a^2+b^2)\displaystyle{}\frac{\int_{B_{R_{0}}} |-2\nabla\chi\cdot\nabla\varphi-\varphi\Delta\chi|^{2}\mathrm{d}x}
{\int_{B_{(1+\delta)R}}|\varphi|^{2}\mathrm{d}x}\\
&\leq&e^{-\frac{\mathcal{K}_{1}}{T-t+\lambda}}(a^2+b^2)\frac{8\|\nabla\chi\|^{2}_{\infty}\int_{B_{R_{0}}} |\nabla\varphi|^{2}\mathrm{d}x+2\|\Delta\chi\|^{2}_{\infty}
\int_{B_{R_{0}}}|\varphi|^{2}\mathrm{d}x}{\int_{B_{(1+\delta)R}}|\varphi|^{2}\mathrm{d}x}
\end{eqnarray*}
 for each $t\in [3T/4,T].$ This, together with  \eqref{3.555} and  \eqref{3.666}, yields that
\begin{equation}\label{3.101010}
\begin{array}{lll}
\displaystyle{}\frac{\int_{B_{R_{0}}} e^{y}|g|^{2}\mathrm{d}x}{\int_{B_{R_{0}}}|f|^{2}\mathrm{d}x}
\leq\displaystyle{}(a^2+b^2)\left(a+\frac{|b|}{a}+\frac{b^4}{a^3}+T^{-2}+\|V\|^{2}_{\infty}\right)
\displaystyle{}e^{-\frac{\mathcal{K}_{1}}{T-t+\lambda}}\frac{\int_{T/2}^{T}\int_{B_{2R_{0}}}|\varphi|^{2}\mathrm{d}x\mathrm{d}s}
{\int_{B_{(1+\delta)R}}|\varphi(x,t)|^{2}\mathrm{d}x}
\end{array}
\end{equation}
for each  $t\in [3T/4,T].$
It follows from \eqref{3.888} and \eqref{3.101010} that
\begin{equation}\label{3.111111}
\begin{array}{lll}
\displaystyle{}\frac{\int_{B_{R_{0}}} e^{y}|g|^{2}\mathrm{d}x}{\int_{B_{R_{0}}}|f|^{2}\mathrm{d}x}
&\leq&\displaystyle{}(a^2+b^2)\left(a+\frac{|b|}{a}+\frac{b^4}{a^3}+T^{-2}+\|V\|^{2}_{\infty}\right)
e^{1+\frac{C_{5}}{(a+b^{2}/a)h_{0}}}e^{-2T\|V\|_{\infty}}e^{-\frac{\mathcal{K}_{1}}{T-t+\lambda}}\\
&\leq&\displaystyle{}(a^2+b^2)\left(a+\frac{|b|}{a}+\frac{b^4}{a^3}+T^{-2}\right)
e^{1+\frac{C_{5}}{(a+b^{2}/a)h_{0}}}e^{-\frac{\mathcal{K}_{1}}{T-t+\lambda}} 
\end{array}
\end{equation}
for  each $t\in [T-h_{0},T].$ Thus, by \eqref{Y2}, \eqref{3.999}, \eqref{Y1} and \eqref{3.111111}, implies that \eqref{3.1yy}.

\textbf{Step 2.} In this step, the aim is to give an upper bound for the term $D_{\ell}$ (i.e.,\eqref{dl} below).

By Lemma \ref{lemma4} and step 1 with $y(t)=\|f(t)\|^{2}$, $C_{0}=1-\varepsilon_{0}$ and $\hbar=\lambda$. Let $\ell>1$ be such that $\ell\lambda<\min (1/2, T /4)$,
			$$
			F_1(t)=e^{-\frac{\mathcal{K}_{1}}{T-t+\lambda}}e^{1+\frac{C_{5}}{(a+b^{2}/a)h_{0}}}
\left(1+\frac{|b|}{a}\right)
\left(a^2+\frac{b^2}{a}+T^{-2}\right)+\|V\|_{\infty}
			$$
			and
$$F_2(t)=(a^2+b^2)\left(a+\frac{|b|}{a}+\frac{b^4}{a^3}+T^{-2}\right)
e^{1+\frac{C_{5}}{(a+b^{2}/a)h_{0}}}e^{-\frac{\mathcal{K}_{1}}{T-t+\lambda}} +\|V\|^{2}_{\infty},$$
knowing that for  each $t\in [T-h_{0},T]$
\begin{equation*}
\left\{\begin{array}{lll}
\left|\frac 12\frac{d}{dt}\|f(t)\|^2+N(t)\|f(t)\|^2\right| \leq \left(\frac{1}{2}N(t)+\frac{1-\varepsilon_{0}}{T-t+\lambda}\right)\|f(t)\|^2+F_1(t)\|f(t)\|^2,\\
\\
N^{\prime}(t) \leq\frac{1+(1-\varepsilon_0)}{T-t+\lambda} N(t)+F_2(t),
\end{array}\right.
\end{equation*}
one can deduce the following interpolation inequality with
			$$
			\|f(x,T-\ell\lambda)\|^{2(1+M_{\ell})} \leq \|f(x,T)\|^{2} \|f(x,T-2\ell\lambda)\|^{2M_{\ell}}(2 \ell+1)^{3 (1-\varepsilon_{0})\left(1+M_{\ell}\right)}\mathrm{e}^{D_{\ell}}
			$$
			that is
\begin{equation*}\label{}
\begin{array}{lll}
\displaystyle{}\left(\int_{B_{R_0}}|u(x,T-\ell\lambda)|^2 e^{-\frac{\gamma_{0}\left|x-x_0\right|^2}{\lambda (\ell+1)}} \mathrm{d} x\right)^{1+M_{\ell}}
&\leq&\displaystyle{}(2 \ell+1)^{3 (1-\varepsilon_{0})\left(1+M_{\ell}\right)} \int_{B_{R_0}}|u(x,T)|^2 e^{-\frac{\gamma_{0}\left|x-x_0\right|^2}{\lambda}} \mathrm{d} x \\
&&\times\displaystyle{}\left(\int_{B_{R_0}}|u(x,T-2 \ell\lambda)|^2 e^{-\frac{\gamma_{0}\left|x-x_0\right|^2}{\lambda(2 \ell+1)}}\mathrm{d}x\right)^{M_{\ell}} e^{D_{\ell}}
\end{array}
\end{equation*}
where
\begin{equation}\label{dl}
D_{\ell}=3\left(1+M_{\ell}\right)\left[2\ell \lambda\int_{T-2 \ell \lambda}^T\left|F_2\right|\mathrm{d} t+\int_{T-2 \ell \lambda}^T\left|F_1\right|\mathrm{d} t\right] \ \mathrm{and}\
M_{\ell} \leq \frac{3(\ell+1)^{1-\varepsilon_{0}}}{1-\left(2/3\right)^{1-\varepsilon_{0}}}.
\end{equation}
Since $$e^{-\frac{\mathcal{K}_{1}}{T-t+\lambda}}=e^{-\frac{\gamma_{0}(\delta+5\delta^2/4)R^2}{ (T-t+\lambda)}} \leq e^{-\frac{\gamma_{0}(\delta+5\delta^2/4)R^2}{2l\lambda}} \ \mathrm{for}\ t \in[T-2 \ell \lambda, T]\  \mathrm{with}\ \ell>1,$$ by the definition of $F_1$ and $F_2$ , we have that
\begin{equation}\label{F1}
\left|F_1(t)\right|\leq e^{-\frac{\gamma_{0}(\delta+5\delta^2/4)R^2}{2l\lambda}}e^{1+\frac{C_{5}}{(a+b^{2}/a)h_{0}}}
\left(1+\frac{|b|}{a}\right)
\left(a^2+\frac{b^2}{a}+T^{-2}\right)+\|V\|_{\infty} \ \mathrm{for}\ t \in[T-2 \ell \lambda, T].
\end{equation}
and
\begin{equation}\label{F2} \left|F_2(t)\right|\leq e^{-\frac{\gamma_{0}(\delta+5\delta^2/4)R^2}{2l\lambda}}(a^2+b^2)\left(a+\frac{|b|}{a}+\frac{b^4}{a^3}+T^{-2}\right)
e^{1+\frac{C_{5}}{(a+b^{2}/a)h_{0}}}+\|V\|^{2}_{\infty} \ \mathrm{for}\ t \in[T-2 \ell \lambda, T].
\end{equation}
Based on \eqref{3.999} and \eqref{3.111111}, we need to assume that $2 \ell \lambda\leq h_{0}$ is true.
By \eqref{F1}, \eqref{F2} and \eqref{1.9}, we conclude that for any  $2\ell \lambda\leq h_0\gamma_{0}(\delta+5\delta^2/4)R^{2}(a+b^{2}/a)/C_5\leq h_0$\\
			$$
			\begin{aligned}
				D_{\ell}&=3\left(1+M_{\ell}\right)\left[2\ell \lambda\int_{T-2 \ell \lambda}^T\left|F_2\right|\mathrm{d} t+\int_{T-2 \ell \lambda}^T\left|F_1\right|\mathrm{d} t\right]\\
				&\leq 3(2l\lambda)^{2}\left(1+M_{\ell}\right)\left[(a^2+b^2)\left(a+\frac{|b|}{a}+\frac{b^4}{a^3}+T^{-2}\right)+\|V\|^{2}_{\infty}\right]\\
		&+6l\lambda\left(1+M_{\ell}\right)\left[\left(1+\frac{|b|}{a}\right)
\left(a^2+\frac{b^2}{a}+T^{-2}\right)+\|V\|_{\infty}\right]\\	
&\leq \mathcal{K}_{2}\left(1+M_{\ell}\right)\left(a^3+b^2+\frac{b^3}{a^2}+\left(\frac{b^2}{a}\right)^3+\left(1+\frac{|b|}{a}\right)T^{-1}+\|V\|^{2/3}_{{\infty}}\right)
			\end{aligned}
			$$
where $\mathcal{K}_{2}:=\mathcal{K}_{2}(R, \delta)>0.$\\

\textbf{Step 3.} We claim that there is $K>1$ such that for any $T>0$ and any $\lambda\leq h_0\gamma_{0}(\delta+5\delta^2/4)R^2(a+b^{2}/a)/(2\ell C_5)$,
\begin{equation}\label{Fa3.15}
\begin{array}{lll}
\displaystyle{}&&\displaystyle{}\left(\int_{B_{(1+3\delta/2)R}}|\varphi(x,T-\ell\lambda)|^2 \mathrm{d} x\right)^{1+K}\\
\\
&\leq& \displaystyle{}e^{\mathcal{K}_{3}\left(1+K\right)\left(a^3+b^2+\frac{b^3}{a^2}+\left(\frac{b^2}{a}\right)^3
+\left(1+\frac{|b|}{a}\right)T^{-1}+\|V\|^{2/3}_{{\infty}}\right)}\left({\int_{T/2}^{T}\int_{B_{2R_{0}}}}|\varphi|^{2}\mathrm dx\mathrm ds\right)^{K}\\
\\
&&\displaystyle{}
\times
\left(\mathrm{e}^{\frac{\gamma_{0}r^2}{2\lambda}}\int_{B_r}|\varphi(x,T)|^2 \mathrm{d} x+\mathrm{e}^{-\frac{\gamma_{0}r^2}{2\lambda}}{\int_{T/2}^{T}\int_{B_{2R_{0}}}}|\varphi|^{2}\mathrm dx\mathrm ds\right).
\end{array}
\end{equation}
By step 2, we can deduce that for any  $2\ell \lambda\leq h_0\gamma_{0}(\delta+5\delta^2/4)R^2(a+b^{2}/a)/C_5$

\begin{equation}\label{3.16a}
\begin{array}{lll}
&&\displaystyle{}\left(\int_{B_{R_0}}|u(x,T-\ell\lambda)|^2 \mathrm{e}^{-\frac{\gamma_{0}\left|x\right|^2}{\lambda (\ell+1)}} \mathrm{d} x\right)^{1+M_{\ell}}\\
&\leq& \displaystyle{}e^{\mathcal{K}_{2}\left(1+M_{\ell}\right)\left(a^3+b^2+\frac{b^3}{a^2}+\left(\frac{b^2}{a}\right)^3
+\left(1+\frac{|b|}{a}\right)T^{-1}+\|V\|^{2/3}_{{\infty}}\right)}(2 \ell+1)^{3 (1-\varepsilon_{0})\left(1+M_{\ell}\right)} \int_{B_{R_0}}|u(x, T)|^2 \mathrm{e}^{-\frac{\gamma_{0}\left|x\right|^2}{\lambda}} \mathrm{d} x \\
&&\times\displaystyle{}\left(\int_{B_{R_0}}|u(x,T-2 \ell\lambda)|^2 \mathrm{e}^{-\frac{\gamma_{0}\left|x\right|^2}{\lambda(2 \ell+1)}}\mathrm{d}x\right)^{M_{\ell}}.
\end{array}
\end{equation}
By Lemma \ref{lemma-1.1}, we have that,
\begin{equation*}
\begin{array}{lll}
\displaystyle{}\int_{B_{R_{0}}}|\varphi(x,T-2l\lambda)|^{2}\mathrm dx
\leq C \left[\left(2a+\frac{|b|}{a}\right)R_{0}^{-2}+(T/2-2l\lambda)^{-1}+\|V\|_{\infty}\right]
\displaystyle{\int_{T/2}^{T}\int_{B_{2R_{0}}}}|\varphi|^{2}\mathrm dx\mathrm ds.
\end{array}
\end{equation*}
This yields that
\begin{equation*}
\begin{array}{lll}
\displaystyle{}\int_{B_{R_0}}|u(x,T-2 \ell\lambda)|^2 \mathrm{e}^{-\frac{\gamma_{0}\left|x\right|^2}{\lambda(2 \ell+1)}}\mathrm{d}x \leq \displaystyle{}C \left[\left(2a+\frac{|b|}{a}\right)R_{0}^{-2}+(2l\lambda)^{-1}+\|V\|_{\infty}\right] {\int_{T/2}^{T}\int_{B_{2R_{0}}}}|\varphi|^{2}\mathrm dx\mathrm ds.
\end{array}
\end{equation*}
This, along with \eqref{3.16a}, implies that
$$
\begin{aligned}
&\left(\int_{B_{R_0}}|u(x,T-\ell\lambda)|^2 \mathrm{d} x\right)^{1+M_{\ell}}\\
\leq& \mathrm{e}^{\frac{\gamma_{0}R_0^2}{\lambda (\ell+1)}{(1+M_{\ell})}}e^{\mathcal{K}_{2}\left(1+M_{\ell}\right)\left(a^3+b^2+\frac{b^3}{a^2}+\left(\frac{b^2}{a}\right)^3
+\left(1+\frac{|b|}{a}\right)T^{-1}+\|V\|^{2/3}_{{\infty}}\right)}(2 \ell+1)^{3 (1-\varepsilon_{0})\left(1+M_{\ell}\right)} \\
&\times \left[\left(2a+\frac{|b|}{a}\right)R_{0}^{-2}+(2l\lambda)^{-1}+\|V\|_{\infty}\right]^{M_{\ell}}\left({\int_{T/2}^{T}\int_{B_{2R_{0}}}}|\varphi|^{2}\mathrm dx\mathrm ds\right)^{M_{\ell}}
\int_{B_{R_0}}|u(x,T)|^2 \mathrm{e}^{-\frac{\gamma_{0}\left|x\right|^2}{\lambda}} \mathrm{d}x.
\end{aligned}
$$
We split $\int_{B_{R_0}}|u(x,T)|^2 \mathrm{e}^{-\frac{\gamma_{0}\left|x\right|^2}{\lambda}} \mathrm{d}x$ into two parts. For any $0<r<R$ such that
			$$
\begin{aligned}
&\int_{B_{R_0}}|u(x,T)|^2 \mathrm{e}^{-\frac{\gamma_{0}\left|x\right|^2}{\lambda}} \mathrm{d}x\\
\leq& \int_{B_r}|\varphi(x,T)|^2 \mathrm{d} x+C\mathrm{e}^{-\frac{\gamma_{0}r^2}{\lambda}}\left[\left(2a+\frac{|b|}{a}\right)R_{0}^{-2}+T^{-1}+\|V\|_{\infty}\right]{\int_{T/2}^{T}\int_{B_{2R_{0}}}}|\varphi|^{2}\mathrm dx\mathrm ds
\end{aligned}
			$$
where $C>0.$ Consequently we have
\begin{equation}\label{a3.17}
\begin{array}{lll}
&&\displaystyle{}\left(\int_{B_{R_0}}|u(x,T-\ell\lambda)|^2 \mathrm{d} x\right)^{1+M_{\ell}}\\
\\
&\leq& \displaystyle{}e^{\mathcal{K}_{2}\left(1+M_{\ell}\right)\left(a^3+b^2+\frac{b^3}{a^2}+\left(\frac{b^2}{a}\right)^3
+\left(1+\frac{|b|}{a}\right)T^{-1}+\|V\|^{2/3}_{{\infty}}\right)}(2 \ell+1)^{3 (1-\varepsilon_{0})\left(1+M_{\ell}\right)}\\
\\
&&\displaystyle{}
\times \left[\left(2a+\frac{|b|}{a}\right)R_{0}^{-2}+(2l\lambda)^{-1}+\|V\|_{\infty}\right]^{M_{\ell}}\left({\int_{T/2}^{T}\int_{B_{2R_{0}}}}|\varphi|^{2}\mathrm dx\mathrm ds\right)^{M_{\ell}}\\
\\
&&\displaystyle{}
\times
\left(\mathrm{e}^{\frac{\gamma_{0}R_0^2}{\lambda (\ell+1)}{(1+M_{\ell})}}\int_{B_r}|\varphi(x,T)|^2 \mathrm{d} x+\mathrm{e}^{-\frac{\gamma_{0}r^2}{\lambda}}\mathrm{e}^{\frac{\gamma_{0}R_0^2}{\lambda (\ell+1)}{(1+M_{\ell})}}{\int_{T/2}^{T}\int_{B_{2R_{0}}}}|\varphi|^{2}\mathrm dx\mathrm ds\right).
\end{array}
\end{equation}
Now, choose $\ell>1$ such that $R_0^2\left(1+M_{\ell}\right)/(\ell+1) \leq r^2/2$ (knowing that $M_{\ell} \leq  \frac{3(\ell+1)^{1-\varepsilon_{0}}}{1-\left(2/3\right)^{1-\varepsilon_{0}}}$ for $\ell>1$ and $0<\varepsilon_{0}<1)$. Therefore, by \eqref{a3.17}, we have that \eqref{Fa3.15}.

\textbf{Step 4.} End of the proof.

By Proposition \ref{lemma-1.3} and Lemma \ref{lemma-1.1}, since $\ell \lambda \leq h_0$,
\begin{equation*}
\begin{array}{lll}
\displaystyle{}\int_{B_{R}}|\varphi(x,T)|^2 \mathrm{d} x&\leq& C\left[a+\frac{|b|}{a}+T^{-1}+\|V\|_{\infty}\right]
\displaystyle{}\int_{T/2}^{T}\int_{B_{(1+2\delta)R}}|\varphi|^{2}\mathrm dx\mathrm ds\\
\\
\displaystyle{}&\leq& C\left[a+\frac{|b|}{a}+T^{-1}+\|V\|_{\infty}\right]e^{-2T\|V\|_{\infty}}\mathrm{e}^{1+\frac{C_5}{(a+b^{2}/a)h_0}}\displaystyle{} \int_{B_{(1+3\delta/2)R}}|\varphi(x,T-\ell\lambda)|^2 \mathrm{d} x\\
\\
\displaystyle{}&\leq& C\left(a+\frac{|b|}{a}+T^{-1}\right)\mathrm{e}^{1+\frac{C_5}{(a+b^{2}/a)h_0}} \displaystyle{}\int_{ B_{(1+3\delta/2)R}}|\varphi(x,T-\ell\lambda)|^2 \mathrm{d}x
\end{array}
\end{equation*}
where $C:=C(R,\delta)>0.$
As a consequence, with the conclusion of step 3, for any $\lambda\leq h_0\gamma_{0}(\delta+5\delta^2/4)R^2(a+b^{2}/a)/(2\ell C_5)$, one obtains
\begin{equation*}
\begin{array}{lll}
\displaystyle{}&&\displaystyle{}\left(\int_{B_{R}}|\varphi(x,T)|^2 \mathrm{d}x\right)^{1+K}\\
&\leq& \displaystyle{}Ce^{\mathcal{K}_{3}\left(1+K\right)\left(a^3+b^2+\frac{b^3}{a^2}+\left(\frac{b^2}{a}\right)^3
+\left(1+\frac{|b|}{a}\right)T^{-1}+\|V\|^{2/3}_{{\infty}}\right)}e^{\frac{C(1+K)}{(a+b^{2}/a)h_{0}}}\\
\\
&&\displaystyle{}
\times \left({\int_{T/2}^{T}\int_{B_{2R_{0}}}}|\varphi|^{2}\mathrm dx\mathrm ds\right)^{K}
\left(\mathrm{e}^{\frac{\gamma_{0}r^2}{2\lambda}}\int_{B_r}|\varphi(x,T)|^2 \mathrm{d} x+\mathrm{e}^{-\frac{\gamma_{0}r^2}{2\lambda}}{\int_{T/2}^{T}\int_{B_{2R_{0}}}}|\varphi|^{2}\mathrm dx\mathrm ds\right).
\end{array}
\end{equation*}
where $C=C(R,\delta)>0$.
On the other hand, for any $$\lambda \in\left(h_0\gamma_{0}[\delta+5\delta^2/4]R^2(a+b^{2}/a)/[2\ell C_5], \ell^{-1} \min (1/2, T/4)\right),$$ one has $1 \leq \mathrm{e}^{-\frac{\gamma_{0}r^2}{2\lambda}} \mathrm{e}^{\frac{\gamma_{0}r^2 \ell}{c_5h_0}}\ \mathrm{where} \ c_{5}:=[\gamma_{0}(\delta+5\delta^{2}/4)R^{2}(a+b^{2}/a)]/C_{5}\in(0,1).$ And for any $\lambda \geq \ell^{-1} \min (1/2, T/4)$, there holds $$1 \leq \mathrm{e}^{\frac{-\gamma_{0}r^2}{2\lambda}}\left(\mathrm{e}^{\gamma_{0}r^2 \ell}+\mathrm{e}^{\frac{2\gamma_{0}r^2 \ell}{T}}\right).$$ Finally, there is $K>1$ such that for any $T>0$ and any $\lambda>0$,
\begin{equation*}
\begin{array}{lll}
\displaystyle{}\left(\int_{B_{R}}|\varphi(x,T)|^2 \mathrm{d}x\right)^{1+K}
&\leq& \displaystyle{}Ce^{C\left(a^3+b^2+\frac{b^3}{a^2}+\left(\frac{b^2}{a}\right)^3
+\left(1+\frac{|b|}{a}\right)T^{-1}+\|V\|^{2/3}_{{\infty}}\right)}\left({\int_{T/2}^{T}\int_{B_{2R_{0}}}}|\varphi|^{2}\mathrm dx\mathrm ds\right)^{K}\\
\\
&&\displaystyle{}
\times e^{\frac{C(1+K)}{(a+b^{2}/a)h_{0}}}
\left(\mathrm{e}^{\frac{\gamma_{0}r^2}{2\lambda}}\int_{B_r}|\varphi(x,T)|^2 \mathrm{d} x+\mathrm{e}^{-\frac{\gamma_{0}r^2}{2\lambda}}{\int_{T/2}^{T}\int_{B_{2R_{0}}}}|\varphi|^{2}\mathrm dx\mathrm ds\right).
\end{array}
\end{equation*}
Next, choose $\lambda>0$ such that
\begin{equation*}
\begin{array}{lll}
&&\displaystyle{}Ce^{C\left(a^3+b^2+\frac{b^3}{a^2}+\left(\frac{b^2}{a}\right)^3
+\left(1+\frac{|b|}{a}\right)T^{-1}+\|V\|^{2/3}_{{\infty}}\right)}e^{\frac{C(1+K)}{(a+b^{2}/a)h_{0}}}
\left({\int_{T/2}^{T}\int_{B_{2R_{0}}}}|\varphi|^{2}\mathrm dx\mathrm ds\right)^{1+K}\mathrm{e}^{-\frac{\gamma_{0}r^2}{2\lambda}}\\
\\
&=&\displaystyle{}\frac{1}{2}\left(\int_{B_{R}}|\varphi(x,T)|^2 \mathrm{d}x\right)^{1+K}.
\end{array}
\end{equation*}
Thus, we have that
\begin{equation*}
\begin{array}{lll}
\displaystyle{}\int_{B_{R}}|\varphi(x,T)|^2 \mathrm{d}x
&\leq&2Ce^{C\left(a^3+b^2+\frac{b^3}{a^2}+\left(\frac{b^2}{a}\right)^3
+\left(1+\frac{|b|}{a}\right)T^{-1}+\|V\|^{2/3}_{{\infty}}\right)}
e^{\frac{C(1+K)}{(a+b^{2}/a)h_{0}}}\\
\\
&&\times\displaystyle{}\left({\int_{T/2}^{T}\int_{B_{2R_{0}}}}|\varphi|^{2}\mathrm dx\mathrm ds\right)^{\frac{1+2K}{2+2K}}\left(\int_{B_r}|\varphi(x,T)|^2 \mathrm{~d} x\right)^{\frac{1}{2+2 K}}.
\end{array}
\end{equation*}
Recall that, by Proposition \ref{lemma-1.3},
\begin{equation*}
\begin{array}{lll} \displaystyle{}e^{\frac{C(1+K)}{(a+b^{2}/a)h_{0}}}&=&\displaystyle{}\left[[1+(a+b^{2}/a)C_{4}]\frac{\int_{T/2}^{T}\int_{B_{2R_{0}}}
|\varphi|^{2}\mathrm{d}x\mathrm{d}t}{\int_{B_{r}}|\varphi(x,T)|^{2}\mathrm{d}x}\right]^{\frac{C(1+K)}{C_3}}\\
\\
&&\times\displaystyle{}\left[\left(e^{[1+2C_{1}(1+\frac{2a+|b|/a}{r^2})]
(1+T^{-1}+\|V\|^{2/3}_{\infty})+\frac{4(a+b^{2}/a)^{-1}C_{3}}{T}
+2T\|V\|_{\infty}}\right)\right]^{\frac{C(1+K)}{C_3}}.
\end{array}
\end{equation*}
Finally, we obtain
\begin{eqnarray*}
\int_{B_{R}}|\varphi(x,T)|^{2}\mathrm{d}x&\leq& e^{\mathcal{K}\left[a^3+b^2+\frac{b^3}{a^2}+\left(\frac{b^2}{a}\right)^3
+\left(1+\frac{|b|}{a}+\frac{a}{a^2+b^2}\right)T^{-1}+\|V\|^{2/3}_{{\infty}}+2T\|V\|_{\infty}\right]}\\
\\
&&\times\left({\int_{T/2}^{T}\int_{B_{2R_{0}}}}|\varphi|^{2}\mathrm dx\mathrm ds\right)^{\beta}
\left(\int_{B_r} |\varphi(x,T)|^{2}\mathrm{d}x\right)^{1-\beta}
\end{eqnarray*}
for some positive constant $\mathcal{K}$, $\beta$ depending on $(R,\delta,r)$.

In summary, we finish the proof of this theorem.\qed

\section{Conclusion and remarks}\label{Conclusions}

This work investigates the quantitative unique continuation properties for the linear Ginzburg-Landau equation. The main novelty of the proposed method consists in constructing an appropriate parabolic frequency function and verifying its logarithmic convexity in time via Carleman-type commutator estimates. Using this convexity property, we directly establish an interpolation inequality for solutions of equation \eqref{1.1} (see Theorem~\ref{Thm1}), which is a key tool for the quantitative description of information propagation. We further generalize and improve upon previous results in multiple aspects. First, unlike \cite{Duan-Wang-Zhang} where the coefficient $b$ is restricted to zero, we allow $b$ to be a general function, thus extending the scope of applicability of the method. Second, building on \cite{Dou-Fu-Liao-Zhu}, we relax the restriction on spatial dimensions and weaken the regularity requirements for coefficients $a$ and $b$. Consequently, our conclusions hold for a wider class of parabolic equations. In particular, while \cite{Dou-Fu-Liao-Zhu} deals with nonlinear settings, the current study focuses on the linear Ginzburg-Landau equation, which paves the way for future studies on nonlinear counterparts.

Finally, we discuss several applications and raise several open questions for future research in the remarks below.

\begin{remark}
We present an application of Theorem~\ref{Thm1} to the null controllability from measurable sets in time (see, e.g., \cite{AEWZ, Phung-Wang}). First, based on Theorem~\ref{Thm1}, we derive an interpolation inequality for solutions of the Ginzburg-Landau equation at a single time instant, analogous to those established for parabolic equations in \cite{AEWZ,Duan-Wang-Zhang,Phung-Wang}. Next, we combine the scaling series method with this interpolation inequality to obtain an observability inequality for the Ginzburg-Landau equation, following the proof of \cite[Theorem 4]{Phung-Wang-Zhang}. Finally, by a standard duality method (see, e.g., \cite{WangGengsheng}) together with the observability inequality, we can obtain the null controllability result.
\end{remark}

\begin{remark}
A natural question is whether the null controllability holds for semilinear Ginzburg-Landau equations in $\mathbb{R}^N$ with control acting on an equidistributed set $\omega$. It is well known that the null controllability was proved for semilinear heat equations in $\mathbb{R}^N$ (see, e.g., \cite{Wang-Zhang}). Roughly speaking, the proof consists of three parts: (i) the null controllability of the linearized system; (ii) a uniform upper bound on the cost of controlling a semilinear heat equation on increasingly large domains; and (iii) a fixed-point argument. Inspired by the ideas in \cite{Wang-Zhang}, analogous results can be obtained.
\end{remark}

\section*{Appendix}\label{app}
\noindent\textbf{Proof\ of\ Lemma~\ref{lemma-1.1}}.
For simplicity we write  $B_{r}:= B_{r}(0)$ and  $B_{R}:=  B_{R}(0).$
Let $\eta\in C_{0}^{\infty}(B_{R})$ satisfy
\begin{equation}\label{2.21111}
0\leq\eta(\cdot)\leq 1 \ \mathrm{in} \ B_{R}, \ \eta(\cdot)=1 \ \mathrm{in} \ B_{r} \
\mathrm{and} \  |\nabla \eta(\cdot)|\leq C(R-r)^{-1}.
\end{equation}
Here and throughout the proof of Lemma~\ref{lemma-1.1},
$C$ denotes a generic positive constant. Let $\xi\in C^{\infty}(\mathbb{R})$ satisfy
\begin{equation}\label{2.31111}
0\leq\xi(\cdot)\leq1, \ |\xi'(\cdot)|\leq C(\tau_{2}-\tau_{1})^{-1} \
\mathrm{in} \ \mathbb{R},
  \end{equation}
\begin{equation}\label{2.41111}
\xi(\cdot)=0 \ \mathrm{in} \ (-\infty, T-\tau_{2}] \ \mathrm{and} \ \xi(\cdot)=1 \
\mathrm{in} \ [T-\tau_{1}, +\infty).
  \end{equation}
Multiplying the first equation of (\ref{1.1}) by $\eta^{2}\xi^{2}\overline{\varphi}$, we have that
\begin{equation}\label{H1}
\partial_{t}\varphi\eta^{2}\xi^{2}\overline{\varphi}-(a+ib)\Delta\varphi\eta^{2}\xi^{2}\overline{\varphi}+V\eta^{2}\xi^{2}\overline{\varphi}\varphi=0
\end{equation}
and
\begin{equation}\label{H2}
\partial_{t}\overline{\varphi}\eta^{2}\xi^{2}\varphi-(a-ib)\Delta\overline{\varphi}\eta^{2}\xi^{2}\varphi+\overline{V}\eta^{2}\xi^{2}\overline{\varphi}\varphi=0.
  \end{equation}
Let $\varphi:=\varphi_{1}+i\varphi_{2}$. By (\ref{H1}), (\ref{H2}) and integrating it over $B_{R}$, we obtain that
\begin{equation*}\label{H3}
\begin{array}{lll}
&&\displaystyle{}\int_{B_{R}}\frac{d}{dt}(\varphi\overline{\varphi})\eta^{2}\xi^{2}(t) \mathrm{d}x+2a\int_{B_{R}} \eta^{2}\xi^{2}\nabla\varphi\nabla\overline{\varphi} \mathrm{d}x\\
\\
&=&\displaystyle{-2}a\int_{B_{R}}(\nabla\varphi_{1}\nabla\eta\varphi_{1}+\nabla\varphi_{2}\nabla\eta\varphi_{2})\xi^{2}\eta\mathrm{d}x
+2b\int_{B_{R}}(\nabla\varphi_{2}\nabla\eta\varphi_{1}-\nabla\varphi_{1}\nabla\eta\varphi_{2})\xi^{2}\eta \mathrm{d}x\\
\\
&&\displaystyle{-}\int_{B_{R}} (V+\overline{V}) \eta^{2}\xi^{2}\varphi\overline{\varphi} \mathrm{d}x.
\end{array}
\end{equation*}
Integrating it over  $(T-\tau_{2},t)$ for  $t\in [T-\tau_{1},T]$
\begin{equation}\label{H3}
\begin{array}{lll}
&&\displaystyle{}\int_{B_{R}}\varphi\overline{\varphi}\eta^{2}\xi^{2} \mathrm{d}x+2a\int_{T-\tau_{2}}^{t}\int_{B_{R}} \eta^{2}\xi^{2}\nabla\varphi\nabla\overline{\varphi} \mathrm{d}x\mathrm{d}s\\
\\
&=&\displaystyle{}2\int_{T-\tau_{2}}^{t}\int_{B_{R}} \eta^{2}\xi\xi'\varphi\overline{\varphi} \mathrm{d}x\mathrm{d}s-2a\int_{T-\tau_{2}}^{t}\int_{B_{R}}(\nabla\varphi_{1}\nabla\eta\varphi_{1}+\nabla\varphi_{2}\nabla\eta\varphi_{2})\xi^{2}\eta\mathrm{d}x\mathrm{d}s\\
\\
&&\displaystyle{+}2b\int_{T-\tau_{2}}^{t}\int_{B_{R}}(\nabla\varphi_{2}\nabla\eta\varphi_{1}-\nabla\varphi_{1}\nabla\eta\varphi_{2})\xi^{2}\eta \mathrm{d}x\mathrm{d}s-\int_{T-\tau_{2}}^{t}\int_{B_{R}} (V+\overline{V}) \eta^{2}\xi^{2}\varphi\overline{\varphi} \mathrm{d}x\mathrm{d}s
\end{array}
\end{equation}
Applying Young's inequality to the term on the right hand of (\ref{H3}), we have that
\begin{eqnarray*}
&&\int_{B_{R}}\eta^{2}\xi^{2}|\varphi|^{2}\mathrm{d}x
+a\int_{T-\tau_{2}}^{t}\int_{B_{R}} \eta^{2}\xi^{2}|\nabla\varphi|^{2} \mathrm{d}x\mathrm{d}s\\
&\leq&\left(2a+\frac{|b|}{a}\right)\int_{T-\tau_{2}}^{t}\int_{B_{R}} |\nabla\eta|^{2}\xi^{2}|\varphi|^{2} \mathrm{d}x\mathrm{d}s
+2\int_{T-\tau_{2}}^{t}\int_{B_{R}} \eta^{2}|\xi'||\varphi|^{2} \mathrm{d}x\mathrm{d}s\\
&&+2\|V\|_{\infty}\int_{T-\tau_{2}}^{t}\int_{B_{R}}\eta^{2}\xi^{2}|\varphi|^{2} \mathrm{d}x\mathrm{d}s.
\end{eqnarray*}
This, along with  (\ref{2.21111})-(\ref{2.41111}), implies that
\begin{eqnarray*}
&&\int_{B_{r}}|\varphi(x,t)|^{2} \mathrm{d}x+a\int_{T-\tau_{1}}^{t}\int_{B_{r}} |\nabla\varphi|^{2} \mathrm{d}x\mathrm{d}s\\
\\
&\leq&C\left[\left(2a+\frac{|b|}{a}\right)(R-r)^{-2}+(\tau_{2}-\tau_{1})^{-1}+\|V\|_{\infty}\right]\\
\\
&&\times\int_{T-\tau_{2}}^{T}\int_{B_{R}}|\varphi|^{2} \mathrm{d}x\mathrm{d}s\ \  \mathrm{for\ each} \ t\in [T-\tau_{1},T].
\end{eqnarray*}
Hence, (\ref{1.2}) follows from the last inequality immediately.
\qed
\\

\noindent\textbf{Proof\ of\ Lemma~\ref{lemma-1.2}}.
For each $r'>0,$ we write $B_{r'}:=  B_{r'}(0)$.  Let $\eta\in C_{0}^{\infty}(B_{4R/3})$ satisfy
\begin{equation}\label{2.61111}
0\leq \eta(\cdot)\leq 1,\ |\nabla\eta(\cdot)|\leq CR^{-1}, \ |\Delta \eta(\cdot)|\leq CR^{-2}\ \mathrm{ in} \ B_{4R/3}
\end{equation}
and
\begin{equation}\label{2.61112}
\eta(\cdot)=1 \ \mathrm{ in} \ B_{R}.
\end{equation}
Here and throughout the proof of Lemma~\ref{lemma-1.2}, $C$ denotes a generic positive constant.
Let $\xi\in C^{\infty}(\mathbb{R})$ satisfy
\begin{equation}\label{2.61113}
0\leq \xi(\cdot)\leq1,\ |\xi'(\cdot)|\leq C\tau^{-1} \ \mathrm{ in} \ \mathbb{R},
\end{equation}
\begin{equation}\label{2.61114}
\xi(\cdot)=0 \ \mathrm{in} \ (-\infty, T-4\tau/3] \ \mathrm{and } \ \xi(\cdot)=1  \ \mathrm{in} \ [T-\tau, +\infty).
\end{equation}
Denote by $z:=\eta\xi\varphi$. It is easy to check that
\begin{equation}\label{2.71111}
\left\{
\begin{array}{lll}
\partial_{t}z-(a+ib)\Delta z=[\eta\xi'-(a+ib)\xi\Delta\eta-V\eta\xi]\varphi-2(a+ib)\xi\nabla\eta\cdot\nabla\varphi&\mathrm{in}\ B_{4R/3}\times (0,T),\\
z=0&\mathrm{on}\ \partial B_{4R/3}\times (0,T),\\
z(T-4\tau/3)=0&\mathrm{in}\ B_{4R/3}.\\
\end{array}
\right.
\end{equation}
On one hand, for each $t\in[T-\tau, T],$ we have that
\begin{equation*}
\begin{array}{lll}
&&-\displaystyle{}\int^{t}_{T-4\tau/3}\int_{B_{4R/3}}(a-ib)\Delta \overline{z} \partial_{s}z+(a+ib)\Delta z \partial_{s}\overline{z}\mathrm{d}x\mathrm{d}s\\
&=&\displaystyle{}a\int_{B_{4R/3}}|\nabla z(x,t)|^{2}\mathrm{d}x-a\int_{B_{4R/3}}|\nabla z(x,T-4\tau/3)|^{2}\mathrm{d}x,
\end{array}
\end{equation*}
which indicates that
$$
a\int_{B_{4R/3}}|\nabla z(x,t)|^{2}\mathrm{d}x\leq
\int^{t}_{T-4\tau/3}\int_{B_{4R/3}}|\partial_{s}z-(a+ib)\Delta z|^2\mathrm{d}x\mathrm{d}s
\;\;\mbox{for each}\;\;t\in[T-\tau, T].
$$
This, along with (\ref{2.61112}) and the second relation of (\ref{2.61114}), implies that
\begin{equation}\label{1.4}
\displaystyle{\max_{t\in[T-\tau,T]}}\int_{B_{R}}|\nabla \varphi(x,t)|^{2}\mathrm{d}x\leq
\frac{1}{a}\int^{T}_{T-4\tau/3}\int_{B_{4R/3}}|\partial_{s}z-(a+ib)\Delta z|^2\mathrm{d}x\mathrm{d}s.
\end{equation}
On the other hand, we have that
\begin{equation}\label{1.5}
\begin{array}{lll}
&&\displaystyle{}\frac{1}{a}\int_{T-4\tau/3}^{T}\int_{B_{4R/3}}
\left|[\eta\xi'-(a+ib)\xi\Delta\eta-V\eta\xi]\varphi-2(a+ib)\xi\nabla\eta\cdot\nabla\varphi\right|^{2}\mathrm{d}x\mathrm{d}t\\
\\
&\leq&\displaystyle{\frac{8}{a}}\int_{T-4\tau/3}^{T}\int_{B_{4R/3}}\left[\eta^2|\xi'|^2+(a^2+b^2)\xi^2|\Delta \eta|^2+\|V\|_{\infty}^{2}\eta^2\xi^2\right]|\varphi|^{2}
\mathrm{d}x\mathrm{d}t\\
\\
&&+\displaystyle{\frac{8}{a}}\int_{T-4\tau/3}^{T}\int_{B_{4R/3}}(a^2+b^2)\xi^2|\nabla \eta|^2|\nabla \varphi|^{2}\mathrm{d}x\mathrm{d}t.
\end{array}
\end{equation}
By (\ref{1.5}), (\ref{2.61111}) and (\ref{2.61113}), we get that
\begin{equation}\label{1.6}
\begin{array}{lll}
&&\displaystyle{\frac{1}{a}}\int_{T-4\tau/3}^{T}\int_{B_{4R/3}}
\left|[\eta\xi'-(a+ib)\xi\Delta\eta-V\eta\xi]\varphi-2(a+ib)\xi\nabla\eta\cdot\nabla\varphi\right|^{2}\mathrm{d}x\mathrm{d}t\\
\\
&\leq&\displaystyle{C}\left[a^{-1}\tau^{-2}+(a^2+b^2)a^{-1}R^{-4}+a^{-1}\|V\|^{2}_{\infty}\right]
\int_{T-4\tau/3}^{T}\int_{B_{4R/3}}|\varphi|^2\mathrm{d}x\mathrm{d}t\\
\\
&&+\displaystyle{}C(a^2+b^2)a^{-1}R^{-2}\int_{T-4\tau/3}^{T}\int_{B_{4R/3}}|\nabla\varphi|^2\mathrm{d}x\mathrm{d}t.
\end{array}
\end{equation}
According to (\ref{1.2}) (where $r, R, \tau_{1}$ and $\tau_{2}$ are replaced by
$4R/3, 2R, 4\tau/3$ and $2\tau$, respectively), it is clear that
$$a\int_{T-4\tau/3}^{T}\int_{B_{4R/3}}| \nabla\varphi|^{2}\mathrm{d}x\mathrm{d}t\leq C\left[\tau^{-1}+\left(2a+\frac{|b|}{a}\right)R^{-2}+\|V\|_{\infty}\right]\int_{T-2\tau}^{T}\int_{B_{2R}}|\varphi|^{2}\mathrm{d}x\mathrm{d}t.$$
This, along with (\ref{1.6}), implies that
\begin{eqnarray*}
&&\displaystyle{}\int_{T-4\tau/3}^{T}\int_{B_{4R/3}}
\left|[\eta\xi'-(a+ib)\xi\Delta\eta-V\eta\xi]\varphi-2(a+ib)\xi\nabla\eta\cdot\nabla\varphi\right|^{2}\mathrm{d}x\mathrm{d}t\\
\\
&\leq&C\left[a^{-1}\tau^{-2}+(a^2+b^2)a^{-1}R^{-4}+a^{-1}\|V\|^{2}_{\infty}\right]\int_{T-2\tau}^{T}\int_{B_{2R}}|\varphi|^{2}\mathrm{d}x\mathrm{d}t\\
\\
&&+C(a^2+b^2)a^{-2}R^{-2}\left[\tau^{-1}+\left(2a+\frac{|b|}{a}\right)R^{-2}+\|V\|_{\infty}\right]\int_{T-2\tau}^{T}\int_{B_{2R}}|\varphi|^{2}\mathrm{d}x\mathrm{d}t\\
\\
&\leq&C\left[1+a^{-1}\tau^{-2}+(a^2+b^2)^{2}a^{-3}R^{-4}+a^{-1}\|V\|^{2}_{\infty}\right]\int_{T-2\tau}^{T}\int_{B_{2R}}|\varphi|^{2}\mathrm{d}x\mathrm{d}t.
\end{eqnarray*}
Hence, (\ref{1.3}) follows from (\ref{1.4}), the first equation of (\ref{2.71111}) and the latter inequality immediately.
\qed

\end{document}